\documentclass[final,hidelinks,onefignum,onetabnum,onealgnum,oneeqnum,onethmnum]{siamart250211}
\usepackage[T2A]{fontenc}
\usepackage[cp1251]{inputenc}
\usepackage[english]{babel}
\usepackage{latexml}
\usepackage{hyperref}
\hypersetup{
	pagebackref=true,
	colorlinks=true,
	linkcolor=blue,
	filecolor=magenta,      
	urlcolor=cyan,
	pdftitle={Abstract Fractional Cauchy Problem: Existence of Propagators and Inhomogeneous Solution Representation},
	pdfauthor={Dmytro Sytnyk, Barbara Wohlmuth}
	pdfpagemode=FullScreen,
}
\usepackage[capitalize,nameinlink]{cleveref}
\usepackage{latexsym,amsfonts,amssymb,amsmath}
\usepackage{graphics,graphicx,longtable,array,euscript,epsfig}
\graphicspath{{./pic/}}
\usepackage{algorithm} % <- Preamble
\usepackage{algorithmic}
\iflatexml

\else
\usepackage{orcidlink} % The next package is needed to type orcid symbol
\usepackage{overpic}
\usepackage[hyperpageref]{backref} % Add backreferences to each citation in the list of references
\renewcommand*{\backrefalt}[4]{%
	\hypersetup{linkcolor=gray}%
	\color{gray}{%
		[%
		\ifcase #1 %
		No citations%
		\or
		Cited on p. #2%
		\else
		Cited on pp. #2%
		\fi
		]%
	}
}
\usepackage[author={Dmytro},draft]{fixme}
\fi
\usepackage{enumitem} % Customize enumeration lists
\usepackage[section]{placeins} %Make \FloatBarrier command at the end of every section

\iflatexml
 
\newtheorem{definition}{Definition}
\renewtheorem{theorem}{Theorem}
\newtheorem{example}{Example} 
\newtheorem{lemma}{Lemma}
\newtheorem{proposition}{Proposition}
\newtheorem{corollary}{Corollary}
\crefname{example}{Example}{Examples} 
\crefname{remark}{Remark}{Remarks} 
\crefname{subsection}{Section}{Sections} 
\crefname{section}{Section}{Sections} 
\else
\newsiamremark{remark}{Remark} 
\theoremheaderfont{\normalfont\bfseries}
\renewtheorem{remark}{Remark} 
\renewtheorem{definition}{Definition}
\renewtheorem{theorem}{Theorem}
 
\renewtheorem{lemma}{Lemma}
\renewtheorem{proposition}{Proposition}
\renewtheorem{corollary}{Corollary}
\crefname{example}{Example}{Examples} 
\crefname{remark}{Remark}{Remarks} 
\crefname{subsection}{Section}{Sections} 
\crefname{section}{Section}{Sections} 

\fi
\newcommand{\N}{\mathbb N}
\newcommand{\Z}{\mathbb Z}

\newcommand{\cL}{\mathcal{L}}

\newcommand{\suml}{\sum\limits}
\newcommand{\intl}{\int\limits}

\newcommand{\e}{\mathrm{e}}

\allowdisplaybreaks

\begin{document}
\title{Abstract Fractional Cauchy Problem: Existence of Propagators and Inhomogeneous Solution Representation}
\iflatexml
	\author{\href{https://orcid.org/0000-0003-3065-4921}{Dmytro Sytnyk}%
		\thanks{%
			Department of Numerical Mathematics, Institute of Mathematics, National Academy of Sciences of Ukraine, Kyiv, 01024, Ukraine; (\href{mailto:sytnik@imath.kiev.ua}{sytnik@imath.kiev.ua}). }%
		\\
		\href{https://orcid.org/0000-0001-6908-6015}{Barbara Wohlmuth}%
		\thanks{Department of Mathematics, Technical University of Munich, Garching, 85748, Germany; (\href{mailto:wohlmuth@ma.tum.de}{wohlmuth@ma.tum.de}).}
	}
\else
	\author{Dmytro Sytnyk\,\orcidlink{0000-0003-3065-4921}%
		\thanks{Department of Mathematics, Technical University of Munich, Garching, 85748, Germany;  (\email{syd@ma.tum.de});
			Department of Numerical Mathematics, Institute of Mathematics, National Academy of Sciences of Ukraine, Kyiv, 01024, Ukraine; (\email{sytnik@imath.kiev.ua}). }%
		\and
		Barbara Wohlmuth\,\orcidlink{0000-0001-6908-6015}%
		\thanks{Department of Mathematics, Technical University of Munich, Garching, 85748, Germany; (\email{wohlmuth@ma.tum.de}).}
	}
\fi
\date{\today}
%=================================================================
% Full title of the paper (Capitalized)
% Author Orchid ID: enter ID or remove command
%\newcommand{\orcidauthorA}{0000-0003-3065-4921} 
%\newcommand{\orcidauthorB}{0000-0001-6908-6015} 
\maketitle

\begin{abstract}
	We consider a Cauchy problem for the inhomogeneous differential equation given in terms of an unbounded linear operator $A$ and the Caputo fractional derivative of order $\alpha \in (0, 2)$ in time.
	The previously known representation of the mild solution to such a problem does not have a conventional variation-of-constants like form, with the propagator derived from the associated homogeneous problem.
	Instead, it relies on the existence of two propagators with different analytical properties.
	This fact limits theoretical and especially numerical applicability of the existing solution representation.
	Here, we propose an alternative representation of the mild solution to the given problem, that consolidates the solution formulas for sub-parabolic, parabolic and sub-hyperbolic equations with a positive sectorial operator $A$ and non-zero initial data.
	The new representation is solely based on the propagator of the homogeneous problem and, therefore, can be considered as a more natural fractional extension of the solution to the classical parabolic Cauchy problem.
	By exploiting a trade-off between the regularity assumptions on the initial data in terms of the fractional powers of $A$ and the regularity assumptions on the right-hand side in time, we show that the proposed solution formula is strongly convergent for $t \geq 0$ under considerably weaker assumptions compared to the standard results from the literature.
	Crucially, the achieved relaxation of space regularity assumptions ensures that the new solution representation is practically feasible for any $\alpha \in (0, 2)$ and is amenable to the numerical evaluation using uniformly accurate quadrature-based algorithms.
\end{abstract}

%\fxnote{The basic premise for this paper is to introduce the new representation and the compare it with the existing one. The main point to highlight is that the new representation is more amenable for approximation because under some mild space regularity assumptions it is valid for }
% Keywords

\iflatexml
\else
	\begin{AMS}
		{34A08, 35R11, 34G10, 35R20, 65L05, 65J08, 65J10}
	\end{AMS}
\fi

\section{Problem setting}
We consider a Cauchy problem for the following fractional order differential equation:
\begin{equation}\label{eq:FCP_DE}
	\partial_t^\alpha u(t) + A u(t) = f(t),\quad  \alpha \in (0,2), \ t  \in (0, T].
\end{equation}
Here,  $\partial_t^\alpha$ denotes the Caputo fractional derivative of order
$\alpha$ with respect to $t$
(see, e.g. \cite[p.\,91]{Kilbas2006})
\begin{equation}\label{eq:FracDeriv_Caputo}
	\partial_t^\alpha u(t)= \frac{1}{\Gamma(n-\alpha)} \intl_0^t(t-s)^{n-\alpha-1}u^{(n)}(s)\, ds,
\end{equation}
with $u^{(n)}(s)$ being the standard integer order derivative of $u(s)$;
$\Gamma(\cdot)$ stands for Euler's Gamma function defined by $\Gamma(x)=\int_0^\infty t^{x-1}e^{-t}dt$ and $n = \lceil \alpha \rceil $ is the smallest integer greater or equal to $\alpha$.
In this work we focus on the linear case of \eqref{eq:FCP_DE}, so the function $f \in C([0, T],X)$ only depends on time and space.
The choice of range $\alpha \in (0, 2)$ for the fractional derivative parameter is motivated by the applications of \eqref{eq:FCP_DE} to modeling of physical phenomena \cite{Schneider1989,Ray2018,DiNezza2012,Mainardi2022}.
%\fxnote{Perhaps it is a good idea to mention some concrete examples of such physical phenomena}
%where $\alpha$ corresponds to the ratio of porosity. 
If $n = 1$, the given differential equation is called sub-parabolic.
Similarly, when $n = 2$, this equation is called sub-hyperbolic.

The time-independent operator $A$ in \eqref{eq:FCP_DE}  is assumed to be closed and linear with the domain $D(A)$ dense in a Banach space $X = X(\|\cdot\|,\Omega)$ %and $\alpha_k \in \mathrm{C}$.
and the spectrum $\mathrm{Sp}(A)$ contained in the sectorial region $\Sigma(\rho_s, \varphi_s)$ of complex plane:
\begin{equation}\label{eq:SpSector}
	\Sigma(\rho_s, \varphi_s) = \left\{ z=\rho_s+\rho \e^{i\theta}:\quad \rho \in [0,\infty), \ \left|\theta\right|< \varphi_s
	\right\}, \quad \rho_s > 0, \  \varphi_s< \frac{\pi}{2}.
\end{equation}
The numbers $\rho_s$ and $\varphi_s$ are called spectral parameters of $A$.
In addition to the assumptions on the spectrum, we suppose that the resolvent of $A$:
$
	R\left(z, A\right) \equiv (zI-A)^{-1}
$
satisfies the bound
\begin{equation}\label{eq:ResSector}
	\left\|(zI-A)^{-1} x\right\| \leq M\frac{\|x\|}{1+\left|z\right|} , \quad M > 0,
\end{equation}
for all $x \in X$ and any fixed $z$ outside the sector $\Sigma(\rho_s, \varphi_s)$ and on its boundary.
Following the established convention \cite{bGavrilyuk2011}, we will call such operator $A$ as strongly positive.
The class of strongly positive operators includes second order elliptic partial differential operators \cite{Haase2006}, as well as more general strongly elliptic pseudo-differential operators defined over the bounded domain \cite{Bilyj2010}.
Spectral parameters $\rho_s$ and $\varphi_s$ can be estimated from the coefficients of the differential expression for such $A$; or using the associated sesquilinear form, if $X$ has additional Hilbert space structure and admits the embedding into a Gelfand triple $(V, X, V^*)$.
Precise results on that behalf are stated in Sections 2.1-2.2 of \cite{Gavrilyuk2004}.

To get an intuitive reasoning regarding the solution to the given equation  let us proceed informally at first.
We apply the Laplace transform $\cL$ to both parts of \eqref{eq:FCP_DE} and evaluate  $\cL\left\{\partial_t^\alpha u\right\}$ using formula (2.140) from \cite{Podlubny1999}:
\begin{equation*}\label{eq:Laplace_FracDeriv_Caputo}
	\cL\left\{\partial_t^\alpha u\right\}(z) = 	z^\alpha \cL\{u\}(z)   - \suml_{k=0}^{n - 1} z^{\alpha-1-k} u^{(k)}(0).
\end{equation*}
This allows us to rewrite the transformed equation as
\begin{equation*}
	z^\alpha \widehat{u}(z)	- \suml_{k=0}^{n-1} z^{\alpha-1-k} u^{(k)}(0) + A \widehat{u}(z) = \widehat{f}(z).
\end{equation*}
Henceforth, the notation $\widehat{u}(z)$, $\widehat{f}(z)$ will be used to denote $\cL\{u\}(z)$, $\cL\{f\}(z)$, respectively.
If the operator $(z^\alpha I + A)$ is invertible, this linear equation admits the solution
\begin{equation}\label{eq:FCP_EQ_LT}
	\widehat{u}(z) = (z^\alpha I + A)^{-1} \left( \suml_{k=0}^{n - 1} z^{\alpha-1-k} u^{(k)}(0) +  \widehat{f}(z)\right),
\end{equation}
which leads us to the representation of $u(t)$:
\[
	u(t)
	= \suml_{k=0}^{n-1} \cL^{-1}\left\{ z^{\alpha-1-k}(Iz^\alpha + A)^{-1} u^{(k)}(0)\right\} +  \cL^{-1}\left\{ (Iz^\alpha + A)^{-1} \widehat{f}
	(z)\right\},
\]
valid as long as the inverse Laplace transform $\cL^{-1}$ of the right-hand side from \eqref{eq:FCP_EQ_LT} exists.
One can employ a well-known property $\widehat{f}\, \widehat{g} = \cL\{f * g\}$ of the convolution
to deduce that
\begin{equation*}
	\begin{aligned}
		\cL^{-1}\left\{ (Iz^\alpha + A)^{-1} \widehat{f}(z) \right \}
		 & = \cL^{-1}\left\{ (Iz^\alpha + A)^{-1}\right\} * \cL^{-1}\{ \widehat{f}(z)\}                          \\
		 & = \frac{1}{2\pi i}\intl_0^t \intl_{\Gamma_I}  \e^{z (t-s)} (z^\alpha I + A)^{-1} \, d z f(s) \, d s . \\
	\end{aligned}
\end{equation*}
By substituting the last expression into the previously obtained formula for $u(t)$, we get
\begin{equation}\label{eq:FCP_Sol}
	\begin{aligned}
		u(t) =
		 & \frac{1}{2\pi i} \suml_{k=0}^{n-1} \intl_{\Gamma_I} \e^{z t} z^{\alpha-1-k} (z^\alpha I + A)^{-1} u^{(k)}(0)\, d z \\
		 & + \frac{1}{2\pi i}\intl_0^t \intl_{\Gamma_I}  \e^{z (t-s)} (z^\alpha I + A)^{-1} f(s)\, d z \, d s .
	\end{aligned}
\end{equation}
The integration contour $\Gamma_I$ in the Bromwich integrals above should be oriented counter-clockwise with respect to the singularities of the integrands.
%chosen in such a way that all the singularities of the integrand remain on the left as $z$ goes along the contour $\Gamma_I$ \cite{} . 
%\fxnote{Note that for $\alpha>2$ the operator $A$ should necessary be bounded \cite[p. 99]{fattorini} \cite[p. 3]{Bazhlekova2001}, hence the assumptions on  are natural for the class of sectorial operators.}%

For the time being, let us assume the existence of such $\Gamma_I$ and that the integrals in formula \eqref{eq:FCP_Sol} are well-defined for the given $A$, $f(t)$ and  $\alpha$.
Formula \eqref{eq:FCP_Sol} determines a unique solution to \eqref{eq:FCP_DE} when the initial values of $u(t)$ are prescribed by the conditions
\begin{equation}\label{eq:FCP_BC}
	\begin{cases}
		u(0) = u_0,                & \text{if $0 < \alpha \leq 1$}, \\
		u(0) = u_0, \ u'(0) = u_1, & \text{if $1< \alpha < 2$}.
	\end{cases}
\end{equation}

Fractional Cauchy problem \eqref{eq:FCP_DE}, \eqref{eq:FCP_BC} will be the main object of the current work.
The theory of such problems for differential operators was developed in the works \cite{Dzherbashian2020,Oldham1974,Schneider1989,Kochubei1990}.
The abstract setting, considered here, has been theoretically studied in \cite{Kochubei1989,Bazhlekova1998} for $\alpha \in (0, 1)$ and then in \cite{Bazhlekova2001} for $\alpha \in [1, 2)$.
Later, various particular cases of problem \eqref{eq:FCP_DE}, \eqref{eq:FCP_BC} were analyzed in the papers \cite{ElBorai2002,Hernandez2010,Li2009,Li2011,Li2015,Li2016,Araya2008,Keyantuo2013}.

The inhomogeneous solution formula for \eqref{eq:FCP_DE}, \eqref{eq:FCP_BC} also has a rich history.
The authors of the pioneering works \cite{Kochubei1989,Bazhlekova1998,Bazhlekova2001} established the solution existence using only the homogeneous version of  \eqref{eq:FCP_Sol}.
Shortly after that, in \cite{ElBorai2002} El-Borai derived the inhomogeneous solution representation for the abstract analogue of \eqref{eq:FCP_DE}, \eqref{eq:FCP_BC}, with the Riemann-Liouville time-fractional derivative in place of $\partial_t^\alpha$.
He used the fractional propagator
$P_\alpha(t) = \intl_0^\infty t^{-\alpha}\suml_{n=0}^\infty \frac{(-s t^{-\alpha})^n}{n! \Gamma(1-\alpha (n + 1) )} \exp{\left( -A s \right)}\, d s$,
which stems from the so-called subordination principle \cite{Pruess2013, Bazhlekova2000,Bazhlekova2001}.
%Such theoretically appealing way to the define $S_\alpha(t)$ in terms of the scalar kernel integral transform of the operator semigroup $\exp{\left( -A t \right)}$ can be traced back to the earlier work of Mainardi \cite{Mainardi1996}, where it was used to model the fractional wave phenomena.
%This particular way to the define the fractional propagator dates back to the earlier works It is based on the subordination principle \cite{Pruess2013, Bazhlekova2000}
In the latter works \cite{Zhou2013,Keyantuo2013} the proposed formula was extended to problem \eqref{eq:FCP_DE}, \eqref{eq:FCP_BC} and its generalizations.
Due to the properties of $P_\alpha(t)$
\cite{Mainardi1996,Gorenflo1998}, its applicability is limited to $\alpha \in (0, 1)$.
Therefore, in order to generalize beyond the sub-parabolic case and justify \eqref{eq:FCP_Sol} one has to adopt an alternative notion of the propagator $S_\alpha(t) = \frac{1}{2\pi i} \intl_{\Gamma_I}  e^{zt}z^{\alpha-1} (z^\alpha I + A)^{-1}  dz$, which was also proposed in \cite{Bazhlekova2001}.
For $\alpha \in (0, 1)$, the formal justification of \eqref{eq:FCP_Sol} with help of $S_\alpha(t)$ was conducted in \cite{Mendes2013}.
This work is notable for containing both global and local existence results.
In \cite{Li2015} the solution formula was extended using the Laplace transform to the case $\alpha \in (1, 2)$.
Actually, the formula provided in \cite[Definition 3.2]{Li2015} expresses
the solution of \eqref{eq:FCP_DE}, \eqref{eq:FCP_BC} solely in terms of $S_\alpha(t)$ and the convolution operations:
$u(t) = S_\alpha(t) u_0 + 1 * S_\alpha(t) u_1 + ( {t^{\alpha-2}}/{\Gamma(\alpha-1)} * S_\alpha * f) (t)$;
but, due to the singularity of $t^{\alpha-2}$ this representation is of practical reference for $\alpha \in (0, 2)$, only when reduced to \eqref{eq:FCP_Sol}.

Meanwhile, the authors of \cite{Hernandez2010} adopted an alternative solution derivation procedure.
Instead of working with the Laplace transformed formula from \eqref{eq:FCP_EQ_LT}, they considered the integral analogue of \eqref{eq:FCP_DE} and used the respective theory of the abstract integral equations \cite{Pruess2013}, which can be applied without specifying the precise notion of a propagator.
This resulted in three distinct mild solution representations \cite[Lemma 1.1]{Hernandez2010}, that are progressively attuned to the local time-regularity of the propagator and $f(t)$.
However, these formulas involve the derivative of the propagator, hence they are also practically unsuitable for the reasons to be explained below.

When the operator $A$ is bounded all reviewed representations are functionally equivalent to a well-known solution formula, viz. \eqref{eq:FCP_ML_sol_repr}, given in terms of the two parameter Mittag-Leffler functions \cite{Kilbas2006,Balachandran2016}.
%The connections between mentioned notions of propagators are given in \cite{Balachandran2019}.

It is worthwhile to note that \eqref{eq:FCP_Sol} might be viewed as an extension of the standard mild solution formula for the abstract first-order Cauchy problem \cite{b_krein_en,Dunford1988,b_cp_Fattorini1984}.
Indeed, when $\alpha = 1$, the contour integrals from \eqref{eq:FCP_Sol}  are reduced to the Dunford-Cauchy representation of
$\exp{\left( -A t \right)} u(0)$ and $\exp{\left( -A (t - s) \right)}f(s)$:
\[
	\exp{\left( -A t \right)}
	= \frac{1}{2\pi i} \intl_{\Gamma_I}  \e^{zt} (zI + A)^{-1} \, d z,
\]
defined in terms of the holomorphic function calculus for sectorial operators \cite{Haase2006}.
If $\alpha \neq 1$, such interpretation of the contour integrals as operator functions can no longer be applied, because the contour $\Gamma_I$ encircles $\mathrm{Sp}(-A)$ and the scalar-part singularity of the first term in \eqref{eq:FCP_Sol}, located at  $z = 0$.

The current study is motivated by an important observation regarding the aforementioned representation of the solution to problem \eqref{eq:FCP_DE}, \eqref{eq:FCP_BC}.
For a sufficiently large $z \in \Gamma_I$,  the norms of the integrands pertaining the homogeneous part of \eqref{eq:FCP_Sol} are asymptotically equal to $\left| e^{zt}\right| |z|^{-1-k}$, $k = 0,1$.
Consequently, they decay at least linearly in $z$, even if $t=0$.
Under the same conditions, the norm of the remaining integrand from \eqref{eq:FCP_Sol} is asymptotically equal to  $\left| e^{z(t-s)}\right| |z|^{-\alpha}$, which leads to the slower-than-linear decay with respect to $z$, for $t=s$ and $\alpha < 1$.
This fact is obviously detrimental to the practical applications of \eqref{eq:FCP_Sol} relying on the quadrature-based numerical evaluation of the solution \cite{McLean2010,Vasylyk2022,Colbrook2022a,Sytnyk2023}.
The same observation applies to the representations from \cite{Hernandez2010}.
The goal of this paper is to present an alternative and more algorithmically relevant form of the mild solution to problem \eqref{eq:FCP_DE}, \eqref{eq:FCP_BC}, valid for any $\alpha \in (0,2)$.

In \cref{sec:FCP_Propagators}, we show that the uniform strong convergence of integrals in \eqref{eq:FCP_Sol} for $\alpha \in (0, 2)$ and $t \in [0,T]$ necessitate the additional space regularity assumptions on the constituents of  \eqref{eq:FCP_DE}, \eqref{eq:FCP_BC}: $u_0 \in D(A^\gamma)$, $f(t)  \in D(A^\kappa)$ with some  $\gamma >0$ and  $\kappa > \frac{1-\alpha}{\alpha}$ (see  \cref{thm:FCP_Saa_sc}).
For $\alpha < 1$, this leads to the considerable reduction of the class of admissible right-hand sides.
In this section, we also further discuss the connection between formula \eqref{eq:FCP_Sol} and other existing solution representations, corresponding propagators, their validity and properties.

Guided by the gathered knowledge,  in \cref{sec:FCP_Sol_Ex_Repr}, we derive an alternative formula for the mild solution to the general inhomogeneous version of problem \eqref{eq:FCP_DE}, \eqref{eq:FCP_BC} with $\alpha \in (0,2)$.
It is the main theoretical result of the work.
The derived formula exhibits the strong convergence for $t \geq 0$, under the minimal space-regularity assumptions for $u_0$ and $f(t)$, akin to those of the standard parabolic Cauchy problem.
This property is achieved at the expense of imposing stricter regularity conditions on the right-hand side in time $f \in W^{1,1}([0, T], X)$, for $\alpha <1$.
To justify the new solution representation, we rely upon the results from the theory of abstract integral equations \cite{Pruess2013}, instead of the usual toolkit from holomorphic function calculus.
The new representation is thoroughly validated in \cite{Sytnyk2023}, where it is used as base for an exponentially convergent numerical method.

\section{Propagators and existing solution representation formulas}\label{sec:FCP_Propagators}
We start by stating several basic facts regarding problem \eqref{eq:FCP_DE}, \eqref{eq:FCP_BC}.
\begin{definition}[\cite{Bazhlekova2001}]\label{def:FCP_Strong_Sol}
	Let  $u_0, u_1 \in X$  and $f \in C([0, T],X)$.
	A function $u \in C([0,T],D(A))$ is called a strong solution of \eqref{eq:FCP_DE}, \eqref{eq:FCP_BC} if
	%$u \in C([0,T],D(A)) \cap C^{n-1}([0,T],X)$, 
	%$\intl_0^{t}  (t-s)^{n-\alpha-1} \left( u(s) - \suml_{k=0}^{n - 1} s^{k} u^{(k)}(0)\right)  d s \in C^{n}([0,T],X) $,	
	\begin{equation}\label{eq:FCP_strong_sol_regularity}
		\begin{aligned}
			 & u \in C^{n-1}([0,T],X),                                                                                                       \\
			 & \intl_0^{t}  (t-s)^{n-\alpha-1} \left( u(s) - \suml_{k=0}^{n - 1} \frac{s^{k}}{k!} u^{(k)}(0)\right)  d s \in C^{n}([0,T],X),
		\end{aligned}
	\end{equation}
	%the following conditions are valid 
	%	\begin{enumerate}
	%		\item $u \in C([0,T],D(A)) \cap C^{n-1}([0,T],X)$,
	%		\item 	$\intl_0^{t}  (t-s)^{n-\alpha-1} \left( u(s) - \suml_{k=0}^{n - 1} s^{k} u^{(k)}(0)\right)  d s \in C^{n}([0,T],X), $
	%	\end{enumerate}
	and \eqref{eq:FCP_DE}, \eqref{eq:FCP_BC} hold.
\end{definition}

\begin{definition}[\cite{Bazhlekova2001}]\label{def:FCP_well_posed}
	Problem \eqref{eq:FCP_DE}, \eqref{eq:FCP_BC} is called well-posed if for any given $x_0, x_1 \in D(A)$, there exists a unique strong solution $u(t, x_0, x_1)$ of \eqref{eq:FCP_DE}, \eqref{eq:FCP_BC} and for arbitrary sequences $x_{kp} \in D(A)$, the fact that $x_{kp} \rightarrow 0$ as $p \rightarrow \infty$, $k=0,1$ implies $u(t, x_{0p}, x_{1p}) \rightarrow 0$, uniformly on compact intervals.
\end{definition}

Next, we introduce the concept of the propagator for the given problem.
It coincides with a solution operator of \eqref{eq:FCP_DE}, \eqref{eq:FCP_BC}, when $u_1 = 0$ and $f(t) \equiv 0$, $t \geq 0$.
\begin{definition}[\cite{Pruess2013}]\label{def:FCP_Sol_Op}
	Let $\alpha \in (0,2)$. A bounded linear operator $S_\alpha(t): X \rightarrow X$ is called the propagator of \eqref{eq:FCP_DE}, \eqref{eq:FCP_BC}  if the following three conditions are satisfied:
	\begin{enumerate}[label=P\arabic*., ref=P\arabic*] \itemsep.1em
		\item $S_\alpha(t)$ is strongly continuous on $X$ for $t \geq 0$ and $S_\alpha(0) = I$.% 
		      \label{def:FCP_Sol_Op_Id}
		\item $S_\alpha(t)D(A) \subset D(A)$ and $AS_\alpha(t) x = S_\alpha(t)Ax$ for all $x \in D(A)$, $t \geq 0$.%
		      \label{def:FCP_Sol_Op_Comm}
		\item $u(t) = S_\alpha(t) u_0$ is the solution of
		      \begin{equation}\label{eq:FCP_Hom_Mild}
			      u(t) = u_0 - \frac{1}{\Gamma(\alpha)}\intl_0^t (t-s)^{\alpha-1} A u(s) ds,
		      \end{equation}
		      for all $u_0 \in D(A)$, $t\geq0$.%
		      \label{def:FCP_Sol_Op_S}
	\end{enumerate}
\end{definition}
In this work, we focus on the mild solution to the given problem.
Let us consider a Volterra integral equation
\begin{equation}\label{eq:FCP_Mild}
	u(t) = \suml_{k=0}^{n-1}u_k t^k - \frac{1}{\Gamma(\alpha)}\intl_0^t (t-s)^{\alpha-1} \left ( A u(s) - f(s) \right ) ds.
	%	\quad
	%	u'_1 = 
	%	\begin{cases}
	%		u_1, &\alpha > 1,\\
	%		0, &\alpha  \leq 1.
	%	\end{cases}
\end{equation}
%where $u'_1 = u_1$, for $\alpha > 1$, and $u'_1 = 0$, otherwise.
Similarly to \eqref{eq:FCP_Hom_Mild}, this equation serves as an integral analogue of problem   \eqref{eq:FCP_DE}, \eqref{eq:FCP_BC}, when $u_1$ and $f(t)$, are not identically equal to zero.
To show this, one needs to apply the Riemann-Liouville integral operator
\begin{equation}\label{eq:FCP_RLInt}
	J_\alpha v (t) = \frac{1}{\Gamma(\alpha)} \intl_0^t(t-s)^{\alpha-1}v(s)\, ds,  \quad v \in L^1([0,T], X)
\end{equation}
to both sides of \eqref{eq:FCP_DE} and evaluate $J_\alpha \partial_t^\alpha$  with help of the formula $J_\alpha \partial_t^\alpha u(t) = u(t) - \suml_{k=0}^{n - 1} \frac{u^{(k)}(0)}{\Gamma(k+1)} t^k $, that is valid if $u(t)$ satisfies regularity assumptions from \cref{def:FCP_Strong_Sol} (see \cite[Section  1.2]{Bazhlekova2001}).
The highlighted analogy suggests us to adopt from \cite{Mendes2013} the following definition of the mild solution.
\begin{definition}\label{def:FCP_Mild_Sol}
	For any given $\alpha \in (0,2)$, $u_0, u_1 \in X$ and $f \in C([0, T],X)$, a function $u \in C([0, T],D(A))$ is called a mild solution of \eqref{eq:FCP_DE}, \eqref{eq:FCP_BC} if it satisfies integral equation \eqref{eq:FCP_Mild}.
\end{definition}

%Just like in the classical integer order case, 
%Naturally, 
Every strong solution of the given problem is also its mild solution.
The backward conjecture is false, since the solution $u(t)$ to \eqref{eq:FCP_Mild} does not satisfy conditions \eqref{eq:FCP_strong_sol_regularity} in general, so the derivative $\partial_t^\alpha u(t)$ might not exist.
The notions of strong and mild solution coincide when $f(t) \equiv 0$ and the second initial condition exhibits additional spatial regularity $u_1 \in D(A)$.
For that reason, the well-posedness of homogeneous version of problem \eqref{eq:FCP_DE}, \eqref{eq:FCP_BC} in the sense of \cref{def:FCP_well_posed} is equivalent to the well-posedness of integral equation \eqref{eq:FCP_Mild} in the sense of \cite[Definition 1.2]{Pruess2013}.
This observation is based on formula \eqref{eq:FCP_ML_sol_repr}, analyticity of the involving solution operators and the bound (2.27) from \cite{Bazhlekova2001}.
Furthermore according to Proposition 1.1 from \cite{Pruess2013},  equation \eqref{eq:FCP_Mild}  is well-posed if and only if the given problem admits the propagator satisfying \cref{def:FCP_Sol_Op}.
%, hence we will use only the later concept.

In order to demonstrate the concrete example of the propagator of \eqref{eq:FCP_DE}, \eqref{eq:FCP_BC}, let us for the moment assume that $A$ is bounded.
Then the inverse Laplace transform of the terms in \eqref{eq:FCP_Sol} can be evaluated explicitly.
It gives rise to the alternative formula for the inhomogeneous solution of \eqref{eq:FCP_DE}, \eqref{eq:FCP_BC}:
\begin{equation}
	\begin{aligned}\label{eq:FCP_ML_sol_repr}
		%		u(t)  = & E_{\alpha,1}(-A t^\alpha) u_0 + t H(\alpha-1) E_{\alpha,2}(-A t^\alpha) u_1 
		u(t)  = & \suml_{k=1}^{n}E_{\alpha,k}(-A t^\alpha) u_{k-1}
		+ \intl_0^t (t - s)^{\alpha - 1} E_{\alpha,\alpha}(-A (t - s)^\alpha) f(s)\, d s,
	\end{aligned}
\end{equation}
valid for the bounded linear operator $A$ and any $u_0, u_1 \in X$.
%\begin{Remark}
%	To avoid confusion with the situation when the term containing $u_1$ is redundant in formula \eqref{eq:FCP_ML_sol_repr} and other mild solution representations studied in the sequel, as of now we assume that the second initial value $u_1 \equiv 0$, if $\alpha \leq 1$.
%\end{Remark}
This representation of $u(t)$ in terms of linear combination of the two-parameter Mittag-Leffler functions $E_{\alpha,\beta}(z) = \suml_{k=0}^{\infty} \frac{z^k}{\Gamma(\alpha k + \beta)}$ is well-known in the field of ordinary fractional differential equations \cite{Podlubny1999,Kilbas2006,Balachandran2016}, where it has far-reaching theoretical and practical applications.
By the direct transformations of the series expansion for $E_{\alpha,1}(-A t^\alpha)$  it is easy to conclude that this operator from \eqref{eq:FCP_ML_sol_repr} satisfies the conditions \ref{def:FCP_Sol_Op_Id} and \ref{def:FCP_Sol_Op_Comm} of \cref{def:FCP_Sol_Op}.
The validity of \ref{def:FCP_Sol_Op_S} follows from the equivalence between \eqref{eq:FCP_Hom_Mild} and the homogeneous version of the given problem, which can be demonstrated by applying $J_\alpha$ to each part of \eqref{eq:FCP_DE}.
Therefore, for any finite $t \geq 0$ and bounded $A$ the function $E_{\alpha,1}(-A t^\alpha)$ is the propagator of \eqref{eq:FCP_DE}, \eqref{eq:FCP_BC} in the sense of \cref{def:FCP_Sol_Op}.
Note that its time derivative of $n$-th order has a singularity at $t=0$, caused by the structure of the kernel in \eqref{eq:FracDeriv_Caputo}.
This is a general property of the fractional order solution to problem \eqref{eq:FCP_DE}, \eqref{eq:FCP_BC}.

Unfortunately, the described notion of propagator has a limited practical utility for the evaluation of \eqref{eq:FCP_ML_sol_repr}, because it is ultimately confined to the case when all operator powers $A^k u_0$, $k \in \N$  are bounded.
Additionally, the convergence of the series for $E_{\alpha,\beta}(z)$ %, $\beta = \alpha,1,2$
might be slow for certain combinations of $\alpha, \beta$, even if $z$ is a scalar \cite{Garrappa2015}.

Next, we would like to extend our analysis to the targeted class of strongly positive operators $A$.
At this point it is instructive to provide some background on the connection between the chosen range $(0, 2)$ for the fractional order parameter $\alpha$ and the aforementioned class of $A$.
As we already mentioned, the bulk of an existing research is devoted to the particular cases of \eqref{eq:FCP_DE} when $A$ is a strongly elliptic linear partial differential or, more generally, pseudo-differential operator with the domain $D(A)$ that is dense in $X$.
These cases are encompassed by the class of strongly positive operators considered in this work \cite{Bilyj2010,fujita}.
Another common property of the solutions to \eqref{eq:FCP_DE}, \eqref{eq:FCP_BC}, with $\alpha \in (0, 2)$,  is their decaying behavior as $t \rightarrow \infty$, typical for the classical first order Cauchy problem ($\alpha =1$).
The case $\alpha = 2$ is essentially different in nature, as it gives rise  to the non-decaying in time propagator $\cos{\sqrt{A}t}$ (see \cite{b_cp_Fattorini1984}), that is well defined only for a class of strip-type operators \cite{stripBade1953, batty2012bounded}.
The class of admissible $A$ becomes even more restrictive when the fractional order parameter $\alpha$ is greater than $2$.
Namely, for such $\alpha$ the operator $A$ should necessary be bounded, otherwise problem \eqref{eq:FCP_DE}, \eqref{eq:FCP_BC} is no longer properly defined \cite[Thm. 2.6]{Bazhlekova2001} (see also \cite[p. 99]{fattorini}).
Therefore, the range $\alpha \in (0, 2)$ is both natural and maximum possible for the targeted class of $A$.

\subsection{Contour-based definition of the propagator}
The extension of solution formula \eqref{eq:FCP_ML_sol_repr} to the class of sectorial operators requires a more general notion of the propagator $S_\alpha(t)$, formulated in terms of the contour representation which is directly compatible with \eqref{eq:FCP_Sol}.
The following lemma is based on the relevant results of \cite[Sec. 2.2]{Bazhlekova2001}.
Parts of its proof are reused in the sequel, hence for convenience of the reader we provide it in full detail.
\begin{lemma}\label{lem:FCP_S_repr}
	Assume that $\alpha \in (0,2)$.
	Let $A$ be a strongly sectorial operator with the spectral parameters $\rho_s$, $\varphi_s$.
	If $\beta \geq 1$, the operator function $S_{\alpha,\beta}(t)$:
	\begin{equation}\label{eq:FCP_SO_cont_repr}
		S_{\alpha,\beta}(t) x  = \frac{1}{2\pi i} \intl_{\Gamma_I}  e^{zt}z^{\alpha-\beta} (z^\alpha I + A)^{-1} x dz
	\end{equation}
	is well defined for any $x \in X$.
	Moreover, under these conditions, $S_\alpha(t) \equiv S_{\alpha,1}(t)$ is the propagator of \eqref{eq:FCP_DE}, \eqref{eq:FCP_BC}.
	The contour $\Gamma_I$ is chosen in such a way that the integral in \eqref{eq:FCP_SO_cont_repr}  is convergent and the curve $z^\alpha$, $z \in \Gamma_I$ is positively oriented with respect to $\mathrm{Sp}(-A)\cup \{0\}$.
\end{lemma}
\begin{proof}
	First, let us check that the operator function $S_{\alpha,\beta}$ is well-defined for any $\beta > 1$ and $t \geq 0$.
	To show this, it is enough to demonstrate that the integral in \eqref{eq:FCP_SO_cont_repr} is uniformly convergent with respect to $t$.
	From definition \eqref{eq:FCP_SO_cont_repr}, one can get
	\[		\| 	S_{\alpha,\beta}(t)x\|
		= \frac{1}{2\pi} \left\| \intl_{\Gamma_I}  {e^{zt}z^{\alpha-\beta}} (z^\alpha I + A)^{-1} x \, dz \right \|
		\leq C_1 \intl_{\Gamma_I}  \left| e^{zt}z^{\alpha-\beta} \right| \left\| R (z^\alpha,-A) x\right \| \, dz.
	\]
	We apply the resolvent bound from \eqref{eq:ResSector} to the last integral and obtain the estimate
	\begin{equation}\label{eq:FCP_SO_norm_est}
		\| 	S_{\alpha,\beta}(t)x\|
		\leq  C \|x\|\intl_{\Gamma_I}  \frac{ |z|^{\alpha-\beta}e^{t\Re{z}} }{1+|z|^\alpha} \, dz
		\leq  C \|x\|\sup\limits_{z \in \Gamma_I} e^{t\Re{z}}\intl_{\Gamma_I}  \frac{ |z|^{\alpha-\beta}}{1+|z|^\alpha} \, dz, \  \forall x \in X.
	\end{equation}
	Here and below,
	$C>0$ denotes a generic finite constant.
	The integral in \eqref{eq:FCP_SO_norm_est} is absolutely convergent for arbitrary $\beta>1$ since the integrand is bounded and decays as $|z|^{-\beta}$ when $z$ is going to infinity on the contour $\Gamma_I$.
	In order to convert \eqref{eq:FCP_SO_norm_est} into the proper bound, we choose $\Gamma_I$ to be a composition of two rays $-\rho e^{\pm i \varphi_s}$, $\rho \in [r,+\infty]$,  that are parallel to the boundary of $-\Sigma(\rho_s, \varphi_s)$, and a longer arc of the circle $re^{is}$, $s \in [-\pi + \varphi_s, \pi - \varphi_s]$.
	Recall that, the function $z^\alpha$ maps the sector $\Sigma(0, \varphi_s)$ into the sector $\Sigma(0, \alpha \varphi_s)$,	thus the condition $\varphi_s < \pi/2$ implies that the curve $z^\alpha$, $z \in \Gamma_I$ is free of self-intersections for all $\alpha \in (0,2)$.
	Then for any given $1< r < \infty$, the last integral in \eqref{eq:FCP_SO_norm_est} can be estimated as
	\[
		\begin{aligned}
			\intl_{\Gamma_I}  \frac{ |z|^{\alpha-\beta}}{1+|z|^\alpha} \, dz & = 2\intl_{r}^{\infty} \frac{ |\rho e^{i \varphi_s}|^{\alpha-\beta}}{1+|\rho e^{i \varphi_s}|^\alpha}  \, d\rho
			+  \intl_{-\pi + \varphi_s}^{\pi-\varphi_s}  \frac{ |r e^{i s}|^{\alpha-\beta}}{1+|r e^{i s}|^\alpha} \, ds                                                                                                                                          \\
			                                                                 & = 2\intl_{r}^{\infty} \frac{\rho^{\alpha-\beta}}{1+ \rho^\alpha} \, d\rho
			+  \intl_{-\pi + \varphi_s}^{\pi - \varphi_s}  \frac{ r^{\alpha-\beta}}{1+r^\alpha} \, ds
			\le 2\intl_{r}^{\infty} \rho^{-\beta} d\rho
			+ \left . \frac{s}{r ^{\beta}} \right|_{-\pi + \varphi_s}^{\pi - \varphi_s}                                                                                                                                                                          \\
			                                                                 & = \frac{2(\pi - \varphi_s)}{r^{\beta}} -  2\left. \frac{\rho^{1-\beta}}{\beta-1}  \right |_{r}^{+\infty}  = \frac{2(\pi - \varphi_s)}{r^{\beta}}  + \frac{2r^{1-\beta}}{\beta-1}.
		\end{aligned}
	\]
	Additionally, the value of $\sup\limits_{z \in \Gamma_I} e^{t\Re{z}}$ from \eqref{eq:FCP_SO_norm_est} is equal to $e^{rt}$, with $r$ being the arc radius of the chosen $\Gamma_I$.
	This observation leads us to the estimate
	\begin{equation}\label{eq:FracIntEstbeta}
		\| 	S_{\alpha,\beta}(t)x\| 	\leq  C \frac{e^{rt}}{r^{\beta-1}}\left( \frac{1}{\beta-1} + \frac{(\pi - \varphi_s)}{r}   \right ) \|x\|, \quad C > 0, \ r > 1,
	\end{equation}
	which clearly shows that $\| S_{\alpha,\beta} x \|$ is bounded for $\beta>1$.
	Furthermore, due to the fact that the dependence on $t$ in \eqref{eq:FCP_SO_norm_est} and  \eqref{eq:FracIntEstbeta} is expressed only by the scalar factor $e^{rt}$, the integral in formula \eqref{eq:FCP_SO_cont_repr} convergences uniformly in $t \in [0,T]$.

	In order to treat $S_{\alpha}(t)$ we recall that for any $x \in D(A^{m+1})$, the resolvent $R(z,A)x$ satisfies the identity (2.25) from \cite{bGavrilyuk2011}:
	\begin{equation}\label{eq:ResCor}
		R(z,A)x = \sum\limits_{k=1}^{m+1} \frac{A^{k-1}}{z^k}x + \frac{1}{z^{m+1}}R(z,A) A^{m+1} x,
		\quad m \in \Z_+.
	\end{equation}
	Now, assume that $t$ is strictly positive.
	Together with \eqref{eq:ResSector}, such an assumption guarantees the strong convergence of the integral in \eqref{eq:FCP_SO_cont_repr} on the contour $\Gamma_I$, described above.
	Then, by applying identity \eqref{eq:ResCor} with $m=0$ to this convergent representation of $S_\alpha(t)$, we obtain
	\[
		\begin{aligned}
			S_\alpha(t)x
			 & =  \frac{1}{2\pi i} \intl_{\Gamma_I}   \frac{\e^{zt}}{z} x dz - \frac{1}{2\pi i} \intl_{\Gamma_I}  \frac{e^{zt}}{z} (z^\alpha I + A)^{-1} Ax \, dz
			= x - S_{\alpha,1+\alpha} (t)Ax,
		\end{aligned}
	\]
	for any $x \in D(A)$.
	The term $ S_{\alpha,1+\alpha}(t) Ax$ satisfies estimate \eqref{eq:FracIntEstbeta} with $\beta = 1 + \alpha$ and any $r>1$.
	Consequently, the last formula can be used to extend the definition of the propagator $S_\alpha(t)$ to the closed interval $t \in [0,T]$, where it remains bounded and strongly continuous.
	The generalization of this result to all $x \in X$ follows from the dense embedding $D(A) \subseteq X$, and the closedness of $A$ \cite{Haase2006}.
	Additionally, bound \eqref{eq:FracIntEstbeta} stipulates that the norm of $ S_{\alpha,1+\alpha}(0) Ax$  can be made arbitrary small by increasing $r$. % from the definition of $\Gamma_I$.
	It implies $S_\alpha(0)x = x$, for any $x \in D(A)$.
	The equality $S_\alpha(0) = I$ on $X$ follows from the strong continuity of $S_\alpha(t)$ established earlier.
	This concludes the proof of \ref{def:FCP_Sol_Op_Id} from Definition \ref{def:FCP_Sol_Op}.

	To show the validity of \ref{def:FCP_Sol_Op_Comm} and \ref{def:FCP_Sol_Op_S}, we proceed by establishing the correspondence between $S_\alpha(t)$ and the solution operator
	\[
		H(a,B,t) = \frac{1}{2\pi i} \intl_{\Gamma_I}  \frac{e^{zt}}{z} \left(I - \widehat{g_\alpha}(z)B \right)^{-1} dz,
	\]
	studied by Pr\"uss \cite{Pruess2013} in the context of a more general integral equation
	\[
		v(t) = v_0 + \intl_0^t g_\alpha(t-s)B v(s) ds.
	\]
	In the case of \eqref{eq:FCP_Hom_Mild}, we have $B=-A$, $g_\alpha(t) =  {t^{\alpha-1}}/{\Gamma(\alpha)}$, $v_0 = u(0)$, so
	\begin{equation*}
		\begin{aligned}
			u(t) = & H\left ( \frac{t^{\alpha-1}}{\Gamma(\alpha)}, -A,  t \right )u(0)
			= \frac{1}{2\pi i} \intl_{\Gamma_I}  \frac{e^{zt}}{z} \left(I + z^{-\alpha}A \right)^{-1} u(0) dz          \\
			=      & \frac{1}{2\pi i} \intl_{\Gamma_I}  e^{zt} z^{\alpha-1}\left(Iz^{\alpha} + A \right)^{-1} u(0) dz,
		\end{aligned}
	\end{equation*}
	according to Theorem 1.3 from \cite{Pruess2013}.
	Now, the validity of \ref{def:FCP_Sol_Op_Comm} and \ref{def:FCP_Sol_Op_S} follows from the results of Theorem 2.1, Example 2.1 of \cite{Pruess2013}, that are applicable because of the bound on $\left(Iz^{\alpha} + A \right)^{-1}$, imposed by \eqref{eq:ResSector}, and the inequality $\alpha \varphi_s < \pi$.
\end{proof}
We point out that the propagator $S_{\alpha,\beta}(t)$  from \cref{lem:FCP_S_repr} is defined only for $\beta \geq 1$, although its representation \eqref{eq:FCP_SO_cont_repr} is formally equivalent to the integral formula for $E_{\alpha,\beta}(-A t^\alpha)$ with other admissible $\beta$ \cite{Bazhlekova2001}.
Our motivation for doing so is to distinguish the precise meaning of propagator from a more general object $E_{\alpha,\beta}(-A t^\alpha)$, while still being able to cover all cases of $\alpha$, $\beta$  relevant to the given problem in its full generality.

The part of the above proof connected with bound \eqref{eq:FracIntEstbeta} will become instrumental in the forthcoming analysis.
We state it as a corollary.
\begin{corollary}\label{thm:FCP_S_prop_b2}
	Let the operator $A$ and the fractional order parameter $\alpha$ fulfill the assumptions of \cref{lem:FCP_S_repr}. If $\beta > 1$, then the integral from definition \eqref{eq:FCP_SO_cont_repr} of the operator function $S_{\alpha,\beta}(t)x$  is strongly convergent for any $t \in [0, T]$  and $x \in X$,  moreover, $S_{\alpha,\beta}(0)x = 0$.
\end{corollary}

By slightly abusing the notation of \cref{lem:FCP_S_repr}, we can rewrite \eqref{eq:FCP_Sol} in a more compact form:
\begin{equation}
	\begin{aligned}\label{eq:FCP_S_sol_repr}
		u(t)  = & \suml_{k=1}^{n}S_{\alpha,k}(t) u_{k-1}
		+ \intl_0^t (t - s)^{\alpha - 1} S_{\alpha,\alpha}(t - s) f(s)\, d s, \quad
	\end{aligned}
\end{equation}
The rigorous proof of this solution formula was provided in \cite{Mendes2013,Li2015}.
Our aim here is to investigate the conditions on $u_0, u_1$ and $f(t)$ needed to maintain the strong convergence of the integral representations for $S_{\alpha}(t) u_0$, $S_{\alpha,2}(t) u_1$ and $S_{\alpha,\alpha}(t-s) f(s)$ for any $t \in [0,T]$.
To understand why this investigation is important let us recall that
$n$-th derivative of solution to the given fractional Cauchy problem is unbounded at $t = 0$.
Such behavior of $u(t)$ has a profound impact on the properties of the finite-difference methods for \eqref{eq:FCP_DE}, \eqref{eq:FCP_BC}.
Their convergence order is typically limited by $n$, even for the multi-step methods (see \cite{Garrappa2015a,Stynes_2019} and the references therein).
Besides that, at each time-step, these methods need to query certain part (or all) of the solution history in order to evaluate $\partial_t^\alpha u(t)$  numerically, which makes them computationally costly and memory constrained \cite{Diethelm2020}.

The mentioned drawbacks of finite-difference methods has put an additional spotlight on the alternative numerical solution evaluation strategies \cite{baffet2017high,Guo2019,Khristenko2021} and, in particular, on the quadrature-based numerical methods \cite{McLean2010,Vasylyk2022,Colbrook2022a,Sytnyk2023}.
Unlike finite-differences, the methods from the later class perform the direct numerical evaluation of  \eqref{eq:FCP_S_sol_repr} via the quadrature of the integral from formula \eqref{eq:FCP_SO_cont_repr}, defining the operator function $S_{\alpha,\beta}(t)$.
Such strategy offers a unique combination of advantages, including multi-level parallelism, concurrent evaluation of solution for different values of $t$ and, most notably, exponential convergence of the approximation for any given $t$, including $t = 0$.
The last property is contingent upon the strong convergence of the contour integral in definition \eqref{eq:FCP_SO_cont_repr}.

We observe that the strong convergence of the integral definition for $S_{\alpha,2}(t)u_1$ with any $u_1 \in X$ and  $t \in [0, T]$, established by the \cref{thm:FCP_S_prop_b2}, can not be proved in general for two other components $S_{\alpha}(t) u_0$ and $S_{\alpha,\alpha}(t-s) f(s)$ of \eqref{eq:FCP_S_sol_repr}.
The convergence of the respective integrals degrade when the scalar argument of $S_{\alpha}(t)$ or $S_{\alpha,\alpha}(t-s)$  approaches $0$, as shown by \eqref{eq:FCP_SO_norm_est}.
In the proof of \cref{lem:FCP_S_repr}, we demonstrated that the strong convergence for $S_{\alpha}(0) u_0$ can be reestablished by redefining the formula for propagator via the identity $S_{\alpha}(t) u_0 = u_0 - S_{\alpha,\alpha+1} (t) Au_0$, valid for $u_0 \in D(A)$.
The next proposition supplies a more refined version of such technique.
\begin{proposition}[\cite{Sytnyk2023}]\label{prop:FCP_res_cor}
	Let  $A$ be the sectorial operator satisfying the assumptions of \cref{lem:FCP_S_repr}.
	If $x \in D(A^{m + \gamma})$ and  $z^\alpha \notin \mathrm{Sp}(A) \cup \{0\}$, then
	for any $\gamma \geq 0$
	\begin{equation}\label{eq:FCP_res_cor_norm_est}
		\left\| z^{\alpha-\beta}(z^\alpha I-A)^{-1} x - \frac{1}{z^\beta}\suml_{k=0}^{m}\frac{A^k x}{z^{\alpha k}}\right \|
		\leq \frac{K(1+M) \left \|A^{m+\gamma} x \right \|}{|z|^{m\alpha + \beta}(1+|z|^\alpha)^\gamma} ,
	\end{equation}
	with some constant  $K > 0$ and $M$ defined by \eqref{eq:ResSector}.
\end{proposition}
This result allows us to relax the space regularity assumptions for $S_{\alpha,1}(t) u_0$ to $u_0 \in D(A^\gamma)$, with some $\gamma>0$, and still maintain the strong convergence of its definition \eqref{eq:FCP_SO_cont_repr} for $t \in [0,T]$.
In the case of $S_{\alpha,\alpha}(t-s)f(s)$, we have a more ambiguous situation, since the integrand in \eqref{eq:FCP_SO_cont_repr} for $\beta = \alpha$ does not have the singularity at $z = 0$.
Thus, if we assume that $\mathrm{Sp}(A)$ is separated from the origin, the integration contour $\Gamma_I$ can be chosen to reside entirely in the same half-plane as $\mathrm{Sp}(-A)$.
The strong convergence of $S_{\alpha,\alpha}(t-s)f(s)$ for the targeted range of $t,s$ can be achieved via \eqref{eq:ResCor} only for such a subclass of sectorial operators.
\begin{lemma}\label{thm:FCP_Saa_sc}
	Assume that operator $A$  satisfies the assumptions of \cref{lem:FCP_S_repr} and $f \in L^1([0, T], X)$.
	The operator function $S_{\alpha,\alpha}(t)$ is well-defined with any $\alpha \in (0,2)$ and the corresponding integral for $S_{\alpha,\alpha}(t-s)f(s)$ is strongly-convergent for any $t \in [0, T]$ and all $s \in [0, t]$, if there exist $\delta > \frac{1-\alpha}{\alpha}$, such that $f(t) \in D(A^\delta)$ and the spectrum $\mathrm{Sp}(A)$ is separated from the origin, i.e. $\rho_s > \rho_0$ for some $\rho_0>0$.
	If  $\mathrm{Sp}(A)$ is arbitrarily close to $0$, then the restriction $f(s) \equiv 0$ is necessary to obtain the strong convergence of $S_{\alpha,\alpha}(t-s)f(s)$ via \eqref{eq:ResCor}.
\end{lemma}
\begin{proof}
	Let us first deal with the situation when $\rho_s \geq \rho_0 > 0$.
	If $\alpha >1$, the sought convergence readily follows from \eqref{eq:FCP_SO_norm_est}.
	Thus, we focus on $\alpha \leq 1$.
	If $f(s) \in D(A^\delta)$,
	\begin{equation}\label{eq:FCP_Saa_ge}
		\begin{aligned}
			\|S_{\alpha,\alpha}(t-s) f(s) \|
			=    & \left\| \frac{1}{2\pi i} \intl_{\Gamma_I}  e^{z(t - s)}(z^\alpha I + A)^{-1} f(s) \, dz \right\|                                                                 \\
			\leq & \frac{1}{2\pi }\left\|\, \intl_{\Gamma_I}  e^{z(t - s)}\left( (z^\alpha I + A)^{-1}
			-
			\suml_{k=0}^{m}\frac{(-A)^k }{z^{\alpha (k + 1)}} \right) f(s) \, dz \right\|                                                                                           \\
			     & +  \left\| \frac{1}{2\pi i} \intl_{\Gamma_I}  e^{z(t-s)} \suml_{k=0}^{m}\frac{(-A)^k f(s)}{z^{\alpha (k + 1)}} \, dz \right\|, \quad m = \lfloor \delta \rfloor.
		\end{aligned}
	\end{equation}
	The integrands from the last term are analytic inside the region encircled by $\Gamma_I$ and remain bounded,  because the norms
	\begin{equation}\label{eq:A_power_subord}
		\|A^kx\| =  \|A^{k-m} A^m x\|
		\leq \frac{M}{2\pi}  \intl_{\Gamma_I} \frac{|z|^{k-m}}{1+|z|} \|A^m x\| \, d z
	\end{equation}
	are bounded for $k < m$.
	Consequently, the second term is equal to zero by the Cauchy integral theorem.
	The first norm from the above bound for $\|S_{\alpha,\alpha}(t-s) f(s) \|$ is estimated via \cref{prop:FCP_res_cor}, yielding
	\begin{equation}\label{eq:FCP_Saa_est}
		\|S_{\alpha,\alpha}(t-s) f(s) \|
		\leq  C\intl_{\Gamma_I} \frac{\left \|A^{\delta} f(s) \right \|}{|z|^{\alpha (m+1)}(1+|z|^\alpha)^{\delta - m}} \, dz, \quad s \in [0, t], \ t \in [0, T].
	\end{equation}
	The strong convergence of the integral in \eqref{eq:FCP_Saa_est} is assured by the condition $\alpha(m+1) + \alpha(\delta - m) > 1$,  that gives us the target inequality for $\delta$.

	Next, we consider the case when the gap between $\mathrm{Sp}(A)$ and $0$ is arbitrarily small.
	Then, similarly to $S_{\alpha}(t)$, we are forced to use the contour $\Gamma_I$ that encircles $-\Sigma(\rho_s, \varphi_s) \cup \{0\}$.
	In such case, the scalar integrals from the second term in \eqref{eq:FCP_Saa_ge} can be regarded as the inverse Laplace transform of $z^{-\alpha (k + 1)}$, so
	\[
		\frac{1}{2\pi i} \intl_{\Gamma_I}  e^{z(t - s)} \suml_{k=0}^{m}\frac{(-A)^k f(s)}{z^{\alpha (k + 1)}} \, dz
		= \suml_{k=0}^{m} \frac{(t - s) ^{\alpha (k + 1)-1}}{\Gamma{({\alpha (k + 1)})}} A^k f(s).
	\]
	By inspecting the leading term of the last sum, we deduce that it is bounded for any $t \in [0, T]$ and $s \in [0, t]$, if and only if $f(s) \equiv 0$.
	This completes the proof.
\end{proof}

This lemma highlights two key deficiencies of solution representation \eqref{eq:FCP_S_sol_repr} or, equivalently,  \eqref{eq:FCP_Sol}.
Firstly, the stable and accurate numerical approximation of the inhomogeneous part from \eqref{eq:FCP_S_sol_repr} is achievable only under additional space-regularity assumption on the right-hand side $f(t)$ of \eqref{eq:FCP_DE}.
For $\alpha < 1$, it is generally more restrictive than the corresponding assumption $u_0 \in D(A^\gamma)$,  $\gamma>0$, needed for the homogeneous part.
If $\alpha < 1/2$, the required space regularity  might simply become incompatible with certain combinations of $f(t)$ and $A$, since the domain of $A^{ \frac{1-\alpha}{\alpha}}$ could be no longer dense in $X$.
Besides, the evaluation of arguments $A^k f(0)$, $k > 1$ from \eqref{eq:ResCor}, needed to enforce the convergence of $S_{\alpha,\alpha}(t-s)f(s)$ with $\alpha < 1/2$, is a challenging numerical problem in itself.
It can not be solved accurately without special precautions when $A$ is a finite-difference or finite-element discretization of some elliptic partial differential operator.
Recall that in the integer order case of problem \eqref{eq:FCP_DE}, \eqref{eq:FCP_BC}, the space regularity assumptions for homogeneous and inhomogeneous parts of the solution are equivalent and quite mild:   $u_0, f(t) \in D(A^\gamma)$,  $\gamma>0$, since the same propagator $\exp{\left( -A t \right)} = S_1(t)$ is used for both parts.

Secondly, the lemma's requirement about the spectrum separation renders the evaluation of $S_{\alpha,\alpha}(t)$  practically unfeasible for an important family of problems with singularly perturbed $A$ \cite{Kadalbajoo2010}.
Furthermore, the existing accuracy theory \cite{McLean2010,bGavrilyuk2011,Vasylyk2022} suggests that a slow convergence of the quadrature approximation of $S_{\alpha,\alpha}(t)$ is to be expected when the spectral parameter $\rho_s$ is positive but small (e.g. \cite{Harrell2004,Roos2008}).

The convergence versus regularity behavior of the formulas that involve $S'_{\alpha,\beta}(t)$ can be analyzed in a similar fashion, albeit this time the use of \cref{prop:FCP_res_cor} should be prepended by the transformation $S'_{\alpha,\beta}(t) = S_{\alpha,\beta-1}(t)$, stemming from \eqref{eq:FCP_SO_cont_repr}.
By such means we get the condition $f(t) \in D(A^\delta)$,  $\delta>1/\alpha$ for the solution representations from Lemma 1.1 of \cite{Hernandez2010}, which is even more restrictive than the corresponding condition $\delta > \frac{1-\alpha}{\alpha}$ for \eqref{eq:FCP_S_sol_repr} or \eqref{eq:FCP_Sol}.

\subsection{Representation of solution via subordination principle}
In the sub-parabolic case, the space regularity assumptions imposed by \cref{thm:FCP_Saa_sc} can be relaxed by introducing yet another notion of the fractional propagator \cite{Pruess2013, Bazhlekova2000,Bazhlekova2001}:
\begin{equation}\label{eq:FCP_SO_subord_repr}
	P_\alpha(t) = \intl_0^\infty \Phi_\alpha(s) \exp{\left( -A st^{\alpha} \right)}\, d s, \quad \Phi_\alpha(z) = \suml_{n=0}^\infty \frac{(-z)^n}{n! \Gamma(1-\alpha (n + 1) )}.
\end{equation}
The so-called \textit{M}-Wright function $\Phi_\alpha(z)$ from \eqref{eq:FCP_SO_subord_repr} obeys the law \cite{Mainardi2010}
\begin{equation}\label{eq:MWright_prop1}
	\intl_0^\infty z^n \Phi_\alpha(z) \, dz = \frac{\Gamma(n+1)}{\Gamma(\alpha n + 1)}, \quad n > 1, \ \alpha \in [0, 1).
\end{equation}
Then, the mild solution of problem \eqref{eq:FCP_DE}, \eqref{eq:FCP_BC} with $\alpha \in (0, 1)$ can be represented as follows \cite{ElBorai2002,Keyantuo2013}
\begin{equation}\label{eq:FCP_P_sol_repr}
	\begin{aligned}
		u(t)  = & P_\alpha(t) u_{0}
		+ \intl_0^t (t - s)^{\alpha - 1} P_{\alpha,\alpha}(t - s) f(s)\, d s,
	\end{aligned}
\end{equation}
where the operator function
$
	P_{\alpha,\alpha}(t) = \alpha \intl_0^\infty s \Phi_\alpha(s) \exp{\left( -A s t^{\alpha}\right)}\, d s
$
is well-defined for all $t \in [0,T]$, owing to  \eqref{eq:MWright_prop1}.
The strong convergence of \eqref{eq:FCP_P_sol_repr} is guaranteed under the minimal space regularity assumptions on both $u_0$ and $f(t)$: $u_0 \in D(A^\gamma)$, $f(t)\in D(A^\kappa)$,  $\gamma,\kappa >0$.
From the application point of view however, the representation of solution to \eqref{eq:FCP_DE}, \eqref{eq:FCP_BC} by means of \eqref{eq:FCP_P_sol_repr} introduces an additional level of complexity associated with the numerical evaluation of the time integrals in $P_\alpha(t)$ and $P_{\alpha,\alpha}(t)$.
For that reason, the numerical potential of \eqref{eq:FCP_SO_subord_repr} has been largely left unexplored, with exception of \cite{Aceto2022}.
The direct approximation of \eqref{eq:FCP_P_sol_repr} might become a worthwhile complement to the numerical methods from \cite{Vasylyk2022,Colbrook2022a,Sytnyk2023}, provided that the underlying integrals can be numerically evaluated in the accurate and efficient manner for small values of $\alpha \in (0, 0.1)$, or when $\alpha \varphi_s > \pi/2$.

The evidence supplied in \cite{Sytnyk2023} shows that all mentioned complications with \eqref{eq:FCP_S_sol_repr} can be avoided, if the solution formula involves only the propagators $S_{\alpha}(t)$, $S_{\alpha,2}(t)$.
The next section is devoted to deriving the alternative representation of the solution to fractional Cauchy problem \eqref{eq:FCP_DE}, \eqref{eq:FCP_BC} that fulfills such property.

\section{New form of mild solution representation}\label{sec:FCP_Sol_Ex_Repr}
According to Proposition 1.2  from \cite{Pruess2013}, the mild solution $u(t)$ of \eqref{eq:FCP_DE}, \eqref{eq:FCP_BC}  can be formally expressed through the variation of parameters formula
\begin{equation}\label{eq:FCP_InhomSol1}
	u(t) = \frac{d}{d t} \intl_0^t S_\alpha(t - s) v(s)\, d s,
\end{equation}
where the function $v(t) =  \suml_{k=0}^{n-1}u_k t^k  + J_\alpha f(t)$ is determined from \eqref{eq:FCP_Mild}.
The next theorem presents a more convenient version of \eqref{eq:FCP_InhomSol1} and formalizes the conditions for its existence.

\begin{theorem}\label{thm:FCP_sol_rep}
	Let $\alpha \in (0,2)$, $A$ be a strongly positive operator with the domain $D(A)$, and
	the spectral parameters $\rho_s$, $\varphi_s$.
	If  $f \in W^{1,1}([0, T], X)$, $u_0, u_1 \in D(A)$, then there exists a unique mild solution $u(t)$ of problem \eqref{eq:FCP_DE}, \eqref{eq:FCP_BC}  that can be represented as follows
	\begin{equation}\label{eq:FCP_InhomSol_rep}
		u(t) = S_\alpha(t) u_0 + S_{\alpha,2}(t) u_1 + J_\alpha S_{\alpha}(t)f(0) + \intl_0^t S_{\alpha}(t - s)   J_\alpha f'(s) d s.
	\end{equation}
	Here, the initial value $u_1 \equiv 0$ for $\alpha \in (0, 1]$  and $S_{\alpha,\beta}(t)$ is defined by \eqref{eq:FCP_SO_cont_repr}.
\end{theorem}
\begin{proof}
	The assumptions imposed on $A$ in the formulation of the theorem together with the results of Lemma \ref{lem:FCP_S_repr} guarantee the existence and well-definiteness of the propagator $S_\alpha(t)$.
	Equation \eqref{eq:FCP_Mild} can be regarded as an abstract integral equation
	\begin{equation}\label{eq:FCP_Mild_abstr}
		u(t) = v(t) -  \left(g_\alpha * A u \right)(t),
	\end{equation}
	with $g_\alpha(t) =  {t^{\alpha-1}}/{\Gamma(\alpha)}$, $v(t) = u_0+ t u_1  + J_\alpha f(t)$.
	Here, the symbol $*$ is again used to denote the convolution $(f * g) (t)= \intl_0^t f(t-s) g(s) \, ds$.

	Let $\alpha > 1$, then the function $v(t)$ is differentiable and
	\begin{align*}
		v'(t)
		 & = u_1 + \frac{\alpha - 1}{\Gamma(\alpha)}\intl_0^t (t-s)^{\alpha-2} f(s)\,ds
		\\ &= u_1 +  g_\alpha(t-s)f(s)\Big|_0^t  +
		\frac{1}{\Gamma(\alpha)}\intl_0^t (t-s)^{\alpha-1} f'(s)\,ds                    \\
		 & = u_1 + g_\alpha(t) f(0) + J_\alpha f'(t).
	\end{align*}
	This formula, together with the theorem's assumptions regarding the regularity of $f(t)$,  ensures that $v \in W^{1,1}([0, T],X)$.
	Hence, we are allowed to apply Proposition 1.2  from \cite{Pruess2013}, which states that the function
	\begin{equation}\label{eq:Pruss_prop_1_2_2}
		u(t) = S_\alpha(t) v(0) + \left( S_{\alpha} * v'\right)(t)
	\end{equation}
	is the unique solution to the Voltera integral equation $u(t) = f(t) - A J_\alpha u (t)$.
	We substitute $v(0)$, $v'(t)$ into \eqref{eq:Pruss_prop_1_2_2} and get
	\begin{equation}\label{eq:FCP_InhomSol_rep1}
		u(t) = S_\alpha(t) u_0 +\left( S_{\alpha} * u_1 \right)(t) + J_\alpha S_{\alpha}(t)f(0) + \left(S_{\alpha} * J_\alpha f'\right)(t).
	\end{equation}
	The terms involving $f(0)$ and $f'(t)$ belong to $D(A)$ for any fixed $t \in [0,T]$ by Proposition 1.1 from \cite{Pruess2013}.
	Property \ref{def:FCP_Sol_Op_Comm} from \cref{def:FCP_Sol_Op} declares that the propagator $S_\alpha(t)$ commutes with $A$, since $u_0, u_1 \in D(A)$ by the theorem's premise.
	Thus, the function $u(t)$ defined by \eqref{eq:FCP_InhomSol_rep1} is the unique solution of \eqref{eq:FCP_Mild_abstr} and of \eqref{eq:FCP_Mild}.
	After showing that, we further simplify the operator acting on the second initial condition $u_1$:
	\begin{align*}
		(S_{\alpha} * u_1) (t) &
		= \frac{1}{2 \pi i } \intl_{\Gamma_I} \intl_0^t  e^{z(t-s)} z^{\alpha-1}(z^\alpha I + A)^{-1} u_1 \, ds \, dz                                                             \\
		                       & = \frac{1}{2\pi i} \intl_{\Gamma_I} \left . \left (-\frac{1}{z} e^{z(t-s)}\right  )\right |_{0}^{t} z^{\alpha-1} (z^\alpha I + A)^{-1} u_1 \, dz \\
		                       & = \frac{1}{2 \pi i } \intl_{\Gamma_I} (e^{zt} - 1)z^{\alpha-2}(z^\alpha I + A)^{-1} u_1 \, dz                                                    \\
		                       & = S_{\alpha,2}(t) u_1  - S_{\alpha,2}(0) u_1 .
	\end{align*}
	The term $S_{\alpha,2}(0)u_1$ is equal to zero by \cref{thm:FCP_S_prop_b2}.
	Therefore, the combination of the last formula and \eqref{eq:FCP_InhomSol_rep1} gives us \eqref{eq:FCP_InhomSol_rep}.

	Now, we assume $\alpha \leq 1$.
	In this case, the operator-independent part $v(t)$ of equation \eqref{eq:FCP_Mild_abstr} takes the form
	$v(t) = u_0  + J_\alpha f(t)$.
	For such $\alpha$ and $v(t)$, the unique solution of this equation is provided by formula (1.11) from \cite{Pruess2013}:
	\[
		u(t) =  S_\alpha(t) u_0 + J_\alpha S_{\alpha}(t)f(0) + \left( g_\alpha * S_{\alpha} * f' \right) (t).
	\]
	As a final step needed to transform this representation into \eqref{eq:FCP_InhomSol_rep}, we use the associativity of convolution:
	\[\left( g_\alpha * S_{\alpha} * f' \right) (t) = \intl_0^t S_{\alpha}(t - s)  J_\alpha f'(s) \, ds,\]
	that holds true in the consequence of Fubini's theorem \cite[Thm. 8.7]{Lang2012}.
	The constructed solution representation involves only the linear bounded and continuous in $t$ operators, hence  $u \in C([0, T],D(A))$ for any given combination of $u_0$, $u_1$ and $f(t)$.
\end{proof}

As we can see from \cref{thm:FCP_sol_rep}, the inhomogeneous part of the newly derived solution representation relies only on the original problem's propagator $S_{\alpha}(t)$.
By \cref{prop:FCP_res_cor}, this implies that the following space regularity assumptions
are needed for the strong convergence of integrals involved in solution representation \eqref{eq:FCP_InhomSol_rep} with $t \geq 0$:
\begin{equation}\label{eq:FCP_sol_sc}
	\exists \gamma,\kappa > \min{\{0,1-\alpha\}}: \ u_0 \in D(A^\gamma),\   f(0), f'(s) \in D(A^\kappa), \quad  s \in (0, T].
\end{equation}
For the affected range $\alpha \in (0, 1]$, formula \eqref{eq:FCP_sol_sc} is more permissive with respect to the space regularity of the function $f$, than the constraints imposed by \cref{thm:FCP_Saa_sc}. %, based on the analysis of solution formula \eqref{eq:FCP_S_sol_repr}. 
In particular, inequality \eqref{eq:A_power_subord} reveals that the domain of $A^\kappa$ with any fixed $\kappa \in (0,1]$ is dense in $X$.
Hence, the strong convergence of the operator function definitions in \eqref{eq:FCP_InhomSol_rep} for $\alpha < 1/2$ is practically realizable for any positive sectorial operator $A$.
This is not always the case for solution formula \eqref{eq:FCP_S_sol_repr}, as discussed in \cref{sec:FCP_Propagators}.

When $\alpha \to 1$, formula \eqref{eq:FCP_InhomSol_rep} recovers the mild solution to the integer order Cauchy problem, which requires precisely the same conditions for the strong convergence as in \eqref{eq:FCP_sol_sc}.
On that account,  the regularity imposed by \eqref{eq:FCP_sol_sc} is as weak as possible;
whereas, formula \eqref{eq:FCP_InhomSol_rep} can be regarded as a more natural fractional extension of the integer order solution than the previously known representation \eqref{eq:FCP_S_sol_repr} or \eqref{eq:FCP_Sol}.
Moreover, if the conditions of \cref{thm:FCP_sol_rep} are met and $u_1 \equiv 0$, then the constructed solution is continuous with respect to $\alpha$ for the whole range  $\alpha \in \left(0, 2\right) $.
The mentioned properties of \eqref{eq:FCP_InhomSol_rep} should prove to be useful for certain parameter identification problems \cite{Li2013a,Zhokh2019} as well as for the final value problems \cite{Jin2015,Kaltenbacher2019}  associated with \eqref{eq:FCP_DE}, \eqref{eq:FCP_BC}.

It is noteworthy to highlight that conditions \eqref{eq:FCP_sol_sc} fit well into an existing solution theory of fractional Cauchy problems.
When $A = - \Delta$ for instance, the condition $u_0 \in D(A^\gamma)$ is equivalent to the assumption $u_0 \in \dot{H}^{2\gamma}(\Omega)$, $\gamma \in (0,1]$, encompassing most of the realistic application scenarios.
Then, the solution $u(t)$ of the homogeneous version of \eqref{eq:FCP_DE}, \eqref{eq:FCP_BC} with such $A$, $u_0$ and $\alpha < 1$ satisfies the energy estimate \cite[Theorem 2.1]{Sakamoto2011}
\[
	\left\| \partial^{m}_t u(t)\right\|_{\dot{H}^{p}(\Omega)}
	\leq c t^{\alpha\frac{2\gamma-p}{2} - m} \left\| u_0 \right\| _{\dot{H}^{2\gamma}(\Omega)},
	\quad 0 \leq p - 2\gamma \leq 2,
\]
indicating that the propagator $S_\alpha(t)$ exhibits a characteristic weak singularity at $t = 0$.
For any other fixed $t>0$, it has only a limited in-space smoothing effect, unlike the integer order propagator $\exp{(-At)}$.
More in-depth review of results on the connections between regularity of  the initial data and the solution to \eqref{eq:FCP_DE}, \eqref{eq:FCP_BC} with such $A$ is given in \cite{jin2019numerical}.
For the abstract solution regularity estimates and their discretized analogues, we refer the reader to \cite{Bazhlekova2001,Bazhlekova2012,Lizama2015,Jin2018b}.

Unsurprisingly, all mentioned benefits of the new solution representation come at a price.
Formula \eqref{eq:FCP_InhomSol_rep} requires more operations to evaluate the solution, when compared to  \eqref{eq:FCP_S_sol_repr}, and it is dependent upon $f'(t)$.
In \cite{Sytnyk2023}, we argue that such trade-off is still worth to consider from the numerical standpoint, because the action of the integer order derivative is local and the values of $J_\alpha f'(t)$, $t \in [0,T]$ can be pre-computed efficiently, even if $f'(t)$ has a singularity at one or both interval endpoints.
Other techniques to evaluate  the Riemann-Liouville integral operator $J_\alpha$ with appealing numerical properties are available from \cite{Baffet2017,Khristenko2021}.
The computational efficiency of \eqref{eq:FCP_InhomSol_rep} is also affected by the order in which the operators $J_\alpha$, $S_{\alpha}(t)$ act on $f(0)$ and $f'(t)$.
The calculated values of $(z^\alpha I+A)^{-1}f(0)$  from \eqref{eq:FCP_SO_cont_repr} can be reused during the quadrature evaluation of $S_\alpha(t) f(0)$ for the different values of $t$.
This explains the particular arrangement of terms in \eqref{eq:FCP_InhomSol_rep}.
More detailed discussion on that matter is provided in the numerical part of the current study \cite{Sytnyk2023}.

If $\alpha \in (1,2)$, we can get rid of $f'(t)$ in \eqref{eq:FCP_InhomSol_rep} using the formula $J_\alpha f'(t) = J_{\alpha-1}f(t) - J_\alpha f(0)$, and exchange $J_{\alpha-1}$ with $S_{\alpha}(t)$ by the associativity of convolution.
In this way, we arrive at the solution representation from \cite{Li2015,Li2016}:
\begin{equation}\label{eq:FCP_InhomSol_rep_Li}
	u(t) = S_\alpha(t) u_0 + \intl_0^t S_{\alpha}(s) u_1\, d s +  J_{\alpha - 1}\left(S_{\alpha} * f\right)\, (t), \quad f \in L^1([0, T],X).
\end{equation}
The results from \cite{Hernandez2010} use a similar derivation procedure as the one employed by us in \cref{thm:FCP_sol_rep}, thus the correspondence between the respective solution formulas is obvious.
The equivalence between \eqref{eq:FCP_InhomSol_rep} and \eqref{eq:FCP_S_sol_repr} is established via \eqref{eq:FCP_InhomSol_rep_Li} and the identity $J_{\alpha-1} S_{\alpha,1}(t) = S_{\alpha,\alpha}(t)$, proved in \cite{Li2015}.

\section*{Conclusions}
This work is devoted to the aspects of solution theory for abstract fractional Cauchy problem \eqref{eq:FCP_DE}, \eqref{eq:FCP_BC} that shed new light on the connections between the practical versatility of the mild solution representation, strong convergence of the involved integral operators and the space-regularity of the problem's initial data $u_0, u_1, f(t)$.
Uniform strong convergence of the integrals from previously known representation \eqref{eq:FCP_Sol} is guaranteed under assumptions $u_0 \in D(A^\gamma)$, $f(t)  \in D(A^\kappa)$ with some  $\gamma >0$ and  $\kappa > {(1-\alpha)}/{\alpha}$.
For sub-parabolic problems, the fulfillment of such assumptions leads to the disproportionately limiting constrains on the space regularity of $f(t)$, that are becoming severe, if $\alpha < 1/2$.
This fact makes formula \eqref{eq:FCP_Sol} unsuitable for the quadrature-based numerical evaluation of the mild solution.
The issue with requiring additional regularity of $f(t)$ is not related to the nature of \eqref{eq:FCP_DE}, \eqref{eq:FCP_BC}, because the homogeneous part of solution, constructed from the native problem's propagator $S_\alpha(t)$ and its integral $\intl_0^t S_\alpha(t-s) u_1 \, ds = S_{\alpha,2}(t) u_1$, remains unaffected.
It is rather related to the particular structure of representation \eqref{eq:FCP_Sol} and its use of the propagator $S_{\alpha,\alpha}(t)$, originally associated with the different fractional Cauchy problem.

To get around this issue, we derived alternative representation \eqref{eq:FCP_InhomSol_rep} of the mild solution to problem \eqref{eq:FCP_DE}, \eqref{eq:FCP_BC},  which  involves only the propagator $S_\alpha(t)$ and remains valid for the same range of fractional order parameter $\alpha \in (0,2)$ as \eqref{eq:FCP_Sol}.
Strong convergence of the integrals in \eqref{eq:FCP_InhomSol_rep} for $t \in [0, T]$ is guaranteed under much more permissive space regularity assumptions \eqref{eq:FCP_sol_sc}, that are actually equivalent to the corresponding assumptions for the classical parabolic Cauchy problem \cite{bGavrilyuk2011}.
In that respect, the space-regularity induced by \eqref{eq:FCP_sol_sc} can be regarded as minimal possible.
Moreover, when $A = - \Delta$, conditions \eqref{eq:FCP_sol_sc} are fully compatible with the existing energy estimates \cite{jin2019numerical}, describing regularity of the solution in the scale of spaces $\dot{H}^{s}(\Omega)$.
This observation gives an additional incentive to use the numerical method from the companion work \cite{Sytnyk2023} for applied problems with a limited spatial regularity.

Even thought the presented results were mostly driven by the numerical applications, the proposed solution representation could also help to get some new theoretical insights.
More refined estimates for inhomogeneous version of the given problem, that are based on the asymptotics of $S_{\alpha}(t)$, is one of them.
An application of the new solution formula to nonlinear extensions of the considered fractional Cauchy problem is another promising direction of research.
Recall that according to general theory the solution of  \eqref{eq:FCP_DE}, \eqref{eq:FCP_BC} with $f(t) = f(t,u(t))$ should satisfy equation \eqref{eq:FCP_InhomSol_rep}.
In this regard, the results of \cref{thm:FCP_sol_rep}, establishing the validity of \eqref{eq:FCP_InhomSol_rep} with no additional space regularity assumptions on $f$,  should be useful.
Such semi-linear extensions of the given problem will be considered in future works.

\iflatexml
{
	\footnotesize

}
\else
{
	\footnotesize
	\bibliographystyle{siamplain}
	\bibliography{references}

\begin{thebibliography}{10}

\bibitem{Aceto2022}
{\sc L.~Aceto and F.~Durastante}, {\em Efficient computation of the wright
  function and its applications to fractional diffusion-wave equations},
  {ESAIM}: Mathematical Modelling and Numerical Analysis, 56 (2022),
  pp.~2181--2196, \url{https://doi.org/10.1051/m2an/2022069}.

\bibitem{Araya2008}
{\sc D.~Araya and C.~Lizama}, {\em Almost automorphic mild solutions to
  fractional differential equations}, Nonlinear Analysis: Theory, Methods \&
  Applications, 69 (2008), pp.~3692--3705,
  \url{https://doi.org/10.1016/j.na.2007.10.004}.

\bibitem{stripBade1953}
{\sc W.~G. Bade}, {\em An operational calculus for operators with spectrum in a
  strip.}, Pacific J. Math., 3 (1953), pp.~257--290,
  \url{https://doi.org/10.2140/pjm.1953.3.257},
  \url{http://projecteuclid.org/euclid.pjm/1103051392}.

\bibitem{baffet2017high}
{\sc D.~Baffet and J.~S. Hesthaven}, {\em High-order accurate adaptive kernel
  compression time-stepping schemes for fractional differential equations},
  Journal of Scientific Computing, 72 (2017), pp.~1169--1195,
  \url{https://doi.org/10.1007/s10915-017-0393-z}.

\bibitem{Baffet2017}
{\sc D.~Baffet and J.~S. Hesthaven}, {\em A kernel compression scheme for
  fractional differential equations}, SIAM Journal on Numerical Analysis, 55
  (2017), pp.~496--520, \url{https://doi.org/10.1137/15M1043960},
  \url{https://arxiv.org/abs/https://doi.org/10.1137/15M1043960}.

\bibitem{Balachandran2016}
{\sc K.~Balachandran, M.~Matar, and J.~Trujillo}, {\em Note on controllability
  of linear fractional dynamical systems}, Journal of Control and Decision, 3
  (2016), pp.~267--279, \url{https://doi.org/10.1080/23307706.2016.1217754}.

\bibitem{batty2012bounded}
{\sc C.~Batty, J.~Mubeen, and I.~V{\"o}r{\"o}s}, {\em {Bounded
  $H\infty$-calculus for strip-type operators}}, Integral Equations and
  Operator Theory, 72 (2012), pp.~159--178,
  \url{https://doi.org/10.1007/s00020-011-1922-z}.

\bibitem{Bazhlekova1998}
{\sc E.~Bazhlekova}, {\em The abstract {Cauchy} problem for the fractional
  evolution equation}, Fractional Calculus and Applied Analysis, 1 (1998),
  pp.~255--270.

\bibitem{Bazhlekova2001}
{\sc E.~Bazhlekova}, {\em Fractional evolution equations in {Banach} spaces},
  PhD thesis, Department of Mathematics and Computer Science, 2001,
  \url{https://doi.org/10.6100/IR549476}.

\bibitem{Bazhlekova2012}
{\sc E.~Bazhlekova}, {\em Strict {$L_p$} solutions for nonautonomous fractional
  evolution equations}, Mathematica Balkanica,  (2012).

\bibitem{Bazhlekova2000}
{\sc E.~G. Bazhlekova}, {\em Subordination principle for fractional evolution
  equations}, Fractional Calculus and Applied Analysis, 3 (2000), pp.~213--230.

\bibitem{Bilyj2010}
{\sc O.~Bilyj, E.~Schrohe, and J.~Seiler}, {\em {$H^\infty$}-calculus for
  hypoelliptic pseudodifferential operators}, in Proc. Amer. Math. Soc,
  vol.~138, American Mathematical Society (AMS), Jan. 2010, pp.~1645--1656,
  \url{https://doi.org/10.1090/S0002-9939-10-10271-8}.

\bibitem{Colbrook2022a}
{\sc M.~J. Colbrook and L.~J. Ayton}, {\em A contour method for time-fractional
  pdes and an application to fractional viscoelastic beam equations}, Journal
  of Computational Physics, 454 (2022), p.~110995,
  \url{https://doi.org/10.1016/j.jcp.2022.110995}.

\bibitem{Mendes2013}
{\sc P.~C.~N. De~Mendes}, {\em Fractional differential equations: a novel study
  of local and global solutions in {Banach} spaces}, PhD thesis, Universidade
  de S{\~a}o Paulo, 2013,
  \url{https://doi.org/10.11606/T.55.2013.tde-06062013-145531}.

\bibitem{DiNezza2012}
{\sc E.~Di~Nezza, G.~Palatucci, and E.~Valdinoci}, {\em Hitchhiker's guide to
  the fractional sobolev spaces}, Bulletin des sciences math{\'e}matiques, 136
  (2012), pp.~521--573, \url{https://doi.org/10.1016/j.bulsci.2011.12.004}.

\bibitem{Diethelm2020}
{\sc K.~Diethelm, R.~Garrappa, and M.~Stynes}, {\em Good (and not so good)
  practices in computational methods for fractional calculus}, Mathematics, 8
  (2020), p.~324, \url{https://doi.org/10.3390/math8030324}.

\bibitem{Dunford1988}
{\sc N.~Dunford and J.~T. Schwartz}, {\em Linear operators. Pt. 1, general
  theory.}, John Wiley \& Sons, New York, 1988.

\bibitem{Dzherbashian2020}
{\sc M.~Dzherbashian and A.~Nersesian}, {\em Fractional derivatives and
  {Cauchy} problem for differential equations of fractional order}, Fractional
  Calculus and Applied Analysis, 23 (2020), pp.~1810--1836,
  \url{https://doi.org/10.1515/fca-2020-0090}.

\bibitem{ElBorai2002}
{\sc M.~M. El-Borai}, {\em Some probability densities and fundamental solutions
  of fractional evolution equations}, Chaos, Solitons \& Fractals, 14 (2002),
  pp.~433--440, \url{https://doi.org/10.1016/s0960-0779(01)00208-9}.

\bibitem{fattorini}
{\sc H.~Fattorini}, {\em Second Order Linear Differential Equations in {Banach}
  Spaces}, North-Holland, Amsterdam, 1985.

\bibitem{b_cp_Fattorini1984}
{\sc H.~O. Fattorini and A.~Kerber}, {\em The {Cauchy} Problem}, Cambridge
  University Press, 1984, \url{https://doi.org/10.1017/CBO9780511662799}.

\bibitem{fujita}
{\sc H.~Fujita, N.~Saito, and T.~Suzuki}, {\em Operator Theory and Numerical
  Methods}, Elsevier, Heidelberg, 2001.

\bibitem{Garrappa2015}
{\sc R.~Garrappa}, {\em Numerical evaluation of two and three parameter
  mittag-leffler functions}, SIAM Journal on Numerical Analysis, 53 (2015),
  pp.~1350--1369, \url{https://doi.org/10.1137/140971191}.

\bibitem{Garrappa2015a}
{\sc R.~Garrappa}, {\em Trapezoidal methods for fractional differential
  equations: Theoretical and computational aspects}, Mathematics and Computers
  in Simulation, 110 (2015), pp.~96--112,
  \url{https://doi.org/10.1016/j.matcom.2013.09.012}.

\bibitem{bGavrilyuk2011}
{\sc I.~Gavrilyuk, V.~Makarov, and V.~Vasylyk}, {\em Exponentially convergent
  algorithms for abstract differential equations}, Frontiers in Mathematics,
  Birkh\"auser/Springer Basel AG, Basel, 2011,
  \url{https://doi.org/10.1007/978-3-0348-0119-5}.

\bibitem{Gavrilyuk2004}
{\sc I.~Gavrilyuk, V.~L. Makarov, and V.~Vasylyk}, {\em A new estimate of the
  sing method for linear parabolic problems including the initial point},
  Computational Methods in Applied Mathematics, 4 (2004), pp.~163--179,
  \url{https://doi.org/10.2478/cmam-2004-0009}.

\bibitem{Gorenflo1998}
{\sc R.~Gorenflo and F.~Mainardi}, {\em Fractional calculus and stable
  probability distributions}, Archives of Mechanics, 50 (1998), pp.~377--388.

\bibitem{Guo2019}
{\sc L.~Guo, F.~Zeng, I.~Turner, K.~Burrage, and G.~E. Karniadakis}, {\em
  Efficient multistep methods for tempered fractional calculus: Algorithms and
  simulations}, {SIAM} Journal on Scientific Computing, 41 (2019),
  pp.~A2510--A2535, \url{https://doi.org/10.1137/18m1230153}.

\bibitem{Haase2006}
{\sc M.~Haase}, {\em Functional Calculus for Sectorial Operators}, Operator
  Theory: Advances and Applications, Birkh\"auser Basel, 1~ed., 2006,
  \url{https://doi.org/10.1007/3-7643-7698-8_2}.

\bibitem{Harrell2004}
{\sc E.~M. {Harrell}}, {\em Geometric lower bounds for the spectrum of elliptic
  {PDEs} with dirichlet conditions in part}, Journal of Computational and
  Applied Mathematics, 194 (2006), pp.~26--35,
  \url{https://doi.org/10.1016/j.cam.2005.06.012}.

\bibitem{Hernandez2010}
{\sc E.~Hern{\'a}ndez, D.~O'Regan, and K.~Balachandran}, {\em On recent
  developments in the theory of abstract differential equations with fractional
  derivatives}, Nonlinear Analysis: Theory, Methods \& Applications, 73 (2010),
  pp.~3462--3471, \url{https://doi.org/10.1016/j.na.2010.07.035}.

\bibitem{jin2019numerical}
{\sc B.~Jin, R.~Lazarov, and Z.~Zhou}, {\em Numerical methods for
  time-fractional evolution equations with nonsmooth data: a concise overview},
  Computer Methods in Applied Mechanics and Engineering, 346 (2019),
  pp.~332--358, \url{https://doi.org/10.1016/j.cma.2018.12.011}.

\bibitem{Jin2018b}
{\sc B.~Jin, B.~Li, and Z.~Zhou}, {\em Discrete maximal regularity of
  time-stepping schemes for fractional evolution equations}, Numerische
  mathematik, 138 (2018), pp.~101--131,
  \url{https://doi.org/10.1007/s00211-017-0904-8}.

\bibitem{Jin2015}
{\sc B.~Jin and W.~Rundell}, {\em A tutorial on inverse problems for anomalous
  diffusion processes}, Inverse problems, 31 (2015), p.~035003,
  \url{https://doi.org/10.1088/0266-5611/31/3/035003}.

\bibitem{Kadalbajoo2010}
{\sc M.~K. Kadalbajoo and V.~Gupta}, {\em A brief survey on numerical methods
  for solving singularly perturbed problems}, Applied mathematics and
  computation, 217 (2010), pp.~3641--3716,
  \url{https://doi.org/10.1016/j.amc.2010.09.059}.

\bibitem{Kaltenbacher2019}
{\sc B.~Kaltenbacher and W.~Rundell}, {\em On an inverse potential problem for
  a fractional reaction--diffusion equation}, Inverse Problems, 35 (2019),
  p.~065004, \url{https://doi.org/10.1088/1361-6420/ab109e}.

\bibitem{Keyantuo2013}
{\sc V.~Keyantuo, C.~Lizama, and M.~Warma}, {\em Spectral criteria for
  solvability of boundary value problems and positivity of solutions of
  time-fractional differential equations}, in Abstract and Applied Analysis,
  vol.~2013, Hindawi, Hindawi Limited, 2013, pp.~1--11,
  \url{https://doi.org/10.1155/2013/614328}.

\bibitem{Khristenko2021}
{\sc U.~Khristenko and B.~Wohlmuth}, {\em Solving time-fractional differential
  equations via rational approximation}, {IMA} Journal of Numerical Analysis,
  (2021), \url{https://doi.org/10.1093/imanum/drac022}.

\bibitem{Kilbas2006}
{\sc A.~Kilbas, H.~Srivastava, and J.~J. Trujillo}, {\em Theory and
  applications of fractional differential equations}, vol.~204 of North-Holland
  mathematics studies, elsevier, Boston, 2006.
\newblock Includes bibliographical references and index.

\bibitem{Kochubei1989}
{\sc A.~N. Kochubei}, {\em The {Cauchy} problem for evolution equations of
  fractional order}, Differentsial'nye Uravneniya, 25 (1989), pp.~1359--1368.
\newblock translation in Differential Equations 25 (1989), no. 8, 967--974
  (1990).

\bibitem{Kochubei1990}
{\sc A.~N. Kochubei}, {\em Fractional-order diffusion}, Differential Equations,
  26 (1990), pp.~485--492.

\bibitem{b_krein_en}
{\sc S.~Krein}, {\em Linear Differential Operators in {Banach} Spaces}, Amer.
  Math. Soc., New York, 1971.

\bibitem{Lang2012}
{\sc S.~Lang}, {\em Real and functional analysis}, vol.~142, Springer Science
  \& Business Media, 2012.

\bibitem{Li2009}
{\sc F.-B. Li, M.~Li, and Q.~Zheng}, {\em Fractional evolution equations
  governed by coercive differential operators}, in Abstract and Applied
  Analysis, vol.~2009, Hindawi, 2009,
  \url{https://doi.org/10.1155/2009/438690}.

\bibitem{Li2013a}
{\sc G.~Li, D.~Zhang, X.~Jia, and M.~Yamamoto}, {\em Simultaneous inversion for
  the space-dependent diffusion coefficient and the fractional order in the
  time-fractional diffusion equation}, Inverse Problems, 29 (2013), p.~065014,
  \url{https://doi.org/10.1088/0266-5611/29/6/065014}.

\bibitem{Li2011}
{\sc K.~Li and J.~Peng}, {\em Fractional abstract {Cauchy} problems}, Integral
  Equations and Operator Theory, 70 (2011), pp.~333--361,
  \url{https://doi.org/10.1007/s00020-011-1864-5}.

\bibitem{Li2015}
{\sc Y.~Li}, {\em Regularity of mild solutions for fractional abstract {Cauchy}
  problem with order $\alpha \in (1, 2)$}, Zeitschrift f{\"u}r angewandte
  Mathematik und Physik, 66 (2015), pp.~3283--3298,
  \url{https://doi.org/10.1007/s00033-015-0577-z}.

\bibitem{Li2016}
{\sc Y.-N. Li, H.-R. Sun, and Z.~Feng}, {\em Fractional abstract {Cauchy}
  problem with order $\alpha \in (1, 2)$}, Dynamics of Partial Differential
  Equations, 13 (2016), pp.~155--177,
  \url{https://doi.org/10.4310/dpde.2016.v13.n2.a4}.

\bibitem{Lizama2015}
{\sc C.~Lizama}, {\em $l_p$-maximal regularity for fractional difference
  equations on {UMD} spaces}, Mathematische Nachrichten, 288 (2015),
  pp.~2079--2092, \url{https://doi.org/10.1002/mana.201400326}.

\bibitem{Mainardi1996}
{\sc F.~Mainardi}, {\em The fundamental solutions for the fractional
  diffusion-wave equation}, Applied Mathematics Letters, 9 (1996), pp.~23--28,
  \url{https://doi.org/10.1016/0893-9659(96)00089-4}.

\bibitem{Mainardi2022}
{\sc F.~Mainardi}, {\em Fractional calculus and waves in linear
  viscoelasticity: an introduction to mathematical models}, World Scientific,
  2022.

\bibitem{Mainardi2010}
{\sc F.~Mainardi, A.~Mura, and G.~Pagnini}, {\em The \textit{M}-wright function
  in time-fractional diffusion processes: A tutorial survey}, International
  Journal of Differential Equations, 2010 (2010), pp.~1--29,
  \url{https://doi.org/10.1155/2010/104505}.

\bibitem{McLean2010}
{\sc W.~McLean and V.~Thom{\'e}e}, {\em Numerical solution via {Laplace}
  transforms of a fractional order evolution equation}, The Journal of Integral
  Equations and Applications,  (2010), pp.~57--94,
  \url{https://doi.org/10.1216/jie-2010-22-1-57}.

\bibitem{Oldham1974}
{\sc K.~Oldham and J.~Spanier}, {\em The fractional calculus theory and
  applications of differentiation and integration to arbitrary order},
  Elsevier, 1974.

\bibitem{Podlubny1999}
{\sc I.~Podlubny}, {\em Fractional Differential Equations: An Introduction to
  Fractional Derivatives, Fractional Differential Equations, to Methods of
  their Solution and some of their Applications}, vol.~198 of Mathematics in
  Science and Engineering, Elsevier Science \& Technology Books, 1999.

\bibitem{Pruess2013}
{\sc J.~Pr{\"u}ss}, {\em Evolutionary integral equations and applications},
  vol.~87 of Modern Birkh{\"a}user Classics, Birkh{\"a}user, 1~ed., 2013.

\bibitem{Ray2018}
{\sc S.~S. Ray and S.~Sahoo}, {\em Generalized fractional order differential
  equations arising in physical models}, Chapman and Hall/CRC, 2018,
  \url{https://doi.org/10.1201/9780429430046}.

\bibitem{Roos2008}
{\sc H.-G. Roos, M.~Stynes, and L.~Tobiska}, {\em Robust Numerical Methods for
  Singularly Perturbed Differential Equations: Convection-Diffusion-Reaction
  and Flow Problems}, Springer Series in Computational Mathematics 24,
  Springer-Verlag Berlin Heidelberg, 2~ed., 2008.

\bibitem{Sakamoto2011}
{\sc K.~Sakamoto and M.~Yamamoto}, {\em Initial value/boundary value problems
  for fractional diffusion-wave equations and applications to some inverse
  problems}, Journal of Mathematical Analysis and Applications, 382 (2011),
  pp.~426--447, \url{https://doi.org/10.1016/j.jmaa.2011.04.058}.

\bibitem{Schneider1989}
{\sc W.~R. Schneider and W.~Wyss}, {\em Fractional diffusion and wave
  equations}, Journal of Mathematical Physics, 30 (1989), pp.~134--144,
  \url{https://doi.org/10.1063/1.528578}.

\bibitem{Stynes_2019}
{\sc M.~Stynes}, {\em Singularities}, in Handbook of Fractional Calculus with
  Applications, G.~E. Karniadakis, ed., vol.~3, De Gruyter, apr 2019,
  pp.~287--306, \url{https://doi.org/10.1515/9783110571684-011}.

\bibitem{Sytnyk2023}
{\sc D.~Sytnyk and B.~Wohlmuth}, {\em Exponentially convergent numerical method
  for abstract {Cauchy} problem with fractional derivative of caputo type},
  Mathematics, 11 (2023), \url{https://doi.org/10.3390/math11102312},
  \url{https://www.mdpi.com/2227-7390/11/10/2312}.

\bibitem{Vasylyk2022}
{\sc V.~Vasylyk, I.~Gavrilyuk, and V.~Makarov}, {\em Exponentially convergent
  method for the approximation of a differential equation with fractional
  derivative and unbounded operator coefficient in a {Banach} space}, Ukrainian
  Mathematical Journal, 74 (2022), pp.~171--185,
  \url{https://doi.org/10.1007/s11253-022-02056-8}.

\bibitem{Zhokh2019}
{\sc A.~Zhokh and P.~Strizhak}, {\em Macroscale modeling the methanol anomalous
  transport in the porous pellet using the time-fractional diffusion and
  fractional brownian motion: A model comparison}, Communications in Nonlinear
  Science and Numerical Simulation, 79 (2019), p.~104922,
  \url{https://doi.org/10.1016/j.cnsns.2019.104922}.

\bibitem{Zhou2013}
{\sc Y.~Zhou, X.~Shen, and L.~Zhang}, {\em Cauchy problem for fractional
  evolution equations with caputo derivative}, The European Physical Journal
  Special Topics, 222 (2013), pp.~1749--1765,
  \url{https://doi.org/10.1140/epjst/e2013-01961-5}.

\end{thebibliography}
}
\fi
\end{document}